\documentclass[a4paper,12pt]{amsart}

\usepackage{float}
\usepackage{euscript,eufrak,verbatim}
\usepackage{graphicx}
\usepackage[usenames]{color}
\usepackage[colorlinks,linkcolor=red,anchorcolor=blue,citecolor=blue]{hyperref}
\usepackage{amsmath}
\usepackage{amsthm}
\usepackage[all]{xy}
\usepackage{graphicx} 
\usepackage{amssymb}

\newtheorem{thm}{Theorem}[section]
\newtheorem{prop}[thm]{Proposition}

\newtheorem{lem}[thm]{Lemma}
\newtheorem{cor}[thm]{Corollary}
\newtheorem{con}[thm]{Condition}

\newtheorem{conj}[thm]{Conjecture}

\def\N{\mathbb{N}}
\def\Z{\mathbb{Z}}
\def\N{\mathbb{N}}
\def\R{\mathbb{R}}
\def\C{\mathbb{C}}

\def\FF{\mathcal{F}}

\def\FFi{\overline{\mathcal{F}}}
\def\SS{\mathcal{S}}

\def\wdim{\text{\rm Widim}}
\def\mdim{\text{\rm mdim}}

\def\BB{\mathcal{B}}
\def\In{\text{\rm Int}^{\Phi}}
\def\Int{\text{\rm Int}}

\def\Pa{\partial^{\Phi}}

\makeatletter 
\@addtoreset{equation}{section}
\numberwithin{equation}{section}

\newcommand{\norm}[1]{\left\lVert#1\right\rVert}

\title[Embedding theorems for discrete dynamical systems and topological flows]{Embedding theorems for discrete dynamical systems and topological flows}
\author{Ruxi Shi
}
\address
{Institute of Mathematics, Polish Academy of Sciences, ul. \'Sniadeckich 8, 00-656 Warszawa, Poland}
\email{rshi@impan.pl}
\begin{document}
	
	\maketitle

\begin{abstract}
In this paper, we investigate the embeddings for topological flows. We prove an embedding theorem for discrete topological system.  Our results apply to suspension flows via constant function, and for this case we show an embedding theorem for suspension flows and give a new proof of Gutman-Jin embedding theorem.
\end{abstract}

\section{Introduction}

 A pair $(X,\R)$ is called a {\it $\R$-flow} ({\it real flow} or {\it topological flow}) if $X$ is a topological compact space and the Ableian group $\R$ acts on $X$ continuously, i.e., $0.x=x$ and $r.(s.x)=(r+s).x$ for all $r,s\in \R$ and $x\in X$. 
In this paper, we are interested in the embedding property of the topological flows. Firstly, we consider a special topological flow called solenoid \cite[V.8.15]{nemyckij1960qualitative}. Let $Y_n=[0, n!]$ such that $0$ is identified with $n!$. 
The {\it solenoid} is defined as 
$$Y=\{(x_n)_{n\in \N}\in \prod_{n\in \N} Y_n: x_n=x_{n+1} \mod n! \},$$
with $\R$-action $\Psi_r: (x_n)_{n\in \N}\mapsto (x_n+r \mod n!)_{n\in \N}$ for all $r\in \R$. 
Clearly, the solenoid $(Y, \Psi)$ is a minimal topological flow. Our first main result are stated as follows.

\begin{prop}\label{prop:extension of solenoid, suspension flow 0}
	Let  $(Z, \Phi)$ be an extension of the solenoid $(Y, \Psi)$. Then $(Z, \Phi)$ is topologically conjugate to the suspension flow under the constant function.
\end{prop}

A pair $(X, T)$ is called a {\it discrete topological (dynamical) system} if $X$ is a topological compact space and $T$ is a homeomorphism on $X$. In other words, $T$ induces a continuous $\Z$-action on $X$ by $n.x= T^n(x)$ for $n\in \Z$ and $x\in X$.

Mean dimension was introduced by Gromov \cite{G}, and was further
investigated by Lindenstrauss and Weiss \cite{LindenstraussWeiss2020MeanTopologicalDimension} as an invariant
of topological dynamical systems. Several applications and interesting relations with other subjects have been studied. In recent years the relation with the so-called embedding problem has
attracted considerable attention. Roughly speaking, the problem is which system $(X, T)$
can be embedded in the shifts on the Hilbert cubes $(([0, 1]^N)^{\Z}, \sigma )$, where $N$
is a natural number and $\sigma$ is the (left) shift on $([0, 1]^N)^{\Z}$. 
Gutman and Tsukamoto \cite{gutman2020embedding}, as well as Gutman, Qiao and Tsukamoto \cite{gutman2019application}, had found a way that the discrete system can be firstly embedded in the space of bounded band-limited real functions and then the Hilbert cubes.

For a discrete system $(X,T)$, it is said to satisfy the {\it marker property} if for each positive integer $N$ there exists an open set $U\subset X$ satisfying that
$$
U\cap T^{-n}U=\emptyset~(\text{for}~0<n<N)~\text{and}~X=\cup_{n\in \Z} T^{n}U.
$$
For example, an extension of an aperiodic minimal system has
the marker property \cite[Lemma 3.3]{L99}. Obviously, the marker property implies the aperiodicity. See \cite{Gut15Jaworski, gutman2017embedding} where the marker property was developed. Gutman, Qiao and Tsukamoto \cite{gutman2019application} proved the embedding theorem of discrete topological system, which asserts that if a topological dynamical system $(X,T)$ satisfies the marker property and the mean dimension $\mdim(X,T)<\frac{a}{2}$
then we can embed it in the shift on $\BB(a)$, where $\BB(a)$ is the space of bounded band-limited real functions (see Section \ref{sec: BB(a)}).
In fact, they proved the result for $\Z^k$-action. In this paper, we show that not only there is an embedding from $(X,T)$ to $\BB(a)$, but also it satisfies that different $\Z$-orbits are embedded in different $\R$-orbits. 

\begin{thm}\label{main thm 2 0}
	If a topological dynamical system $(X,T)$ satisfies the marker property and $\mdim(X,T)<\frac{a}{2}$
	then we can embed it in the shift on $\BB(a)$ via $h$ and $h$ satisfies that if there exists $x,x'\in X$ and $r\in \R$ such that $\Phi_r(h(x))=h(x')$ then $r\in \Z$ and $x'=T^rx$.
\end{thm}
Actually, using the same proof, we can prove the above result for $\Z^k$-action. Since we focus on $\R$-action (as well as $\Z$-action) in the current paper, we only give the proof for $\Z$-action.

Using Theorem \ref{main thm 2 0} and Proposition \ref{prop:extension of solenoid, suspension flow 0}, we give a new proof of Gutman-Jin embedding
theorem \cite[Theorem 5.1]{gutmanjin2020realflow} as follows.
\begin{cor}\label{cor:embedding gutman}
	Let $(X,\Phi)$ be an extension of the solenoid. Suppose that $\mdim(X, \Phi)<a/2$. Then $(X,\Phi)$ can be embedded in $B(a)$.
\end{cor}

We organize the paper as follows. Firstly, in Section \ref{sec:Preliminary}, we recall basic notions of discrete topological systems and topological flows. In Section \ref{sec:solenoid}, we investigate the properties of the solenoid and its extensions. In Section \ref{sec:Mean dimension of R-flow}, we recall basic properties of mean dimension of topological flows. In Section \ref{sec:Strongly embedding Theorem for discrete systems}, we prove Theorem \ref{main thm 2 0}.  In Section \ref{sec:Embedding Theorem for suspension flow}, we show an embedding theorem for suspension flows and consequently prove Corollary \ref{cor:embedding gutman}.   Finally, in Section \ref{sec:Further discussion}, we discuss several open problems.
\subsection*{Notations}
\begin{itemize}
	\item $\R_{>0}=(0, \infty)$, $\R_{\ge 0}=[0, \infty)$ and $\R_{\le 0}=( -\infty, 0]$.
	\item Denote a topological flow by $(X, \Phi)=(X, (\Phi_t)_{t\in \R})$ or $(X, \R)$ whenever we do not emphasis the precise $\R$-action on $X$.
\end{itemize}

\section{Preliminary}\label{sec:Preliminary}

In this section, we recall several notions of discrete topological systems and topological flows.
\subsection{Suspension flow and extension}
Let $(Z, \rho)$ be a compact metric space and $T:Z\to Z$ a homeomorphism. Let $f: Z\to \mathbb{R}_{>0}$ be a continuous map. The {\it suspension flow} of $T$ under $f$, written by $(Z_f, T_f)$, is the flow $(T_t)_{t\in \R}$ on the space
$$
Z_{f}:=\{(x,t): 0\le t\le f(x), x\in Z \}/ (x,f(x))\sim (Tx,0)
$$
induced by the time translation $T_t$ on $Z\times \mathbb{R}$ defined by $T_t(x,s)=(x, t+s)$.

\subsection{Cross-section}
	A \textit{cross-section} of time $\eta>0$ is a subset $S\subset X$ such that the restriction of $\Phi$ on $S\times [-\eta, \eta]$ is one-to-one and onto its image. Moreover, it is said to be \textit{global} if there is $\xi>0$ such that $\Phi(S\times [-\xi, \xi])=X$. The \textit{flow interior} of the cross-section $S$ is defined by
	$$
	\In(S):=\text{Int}(\Phi_{(-\gamma, \gamma)}(S))\cap S,
	$$ 
	for any $0<\gamma<\eta$.  The \textit{flow boundary} of $S$ is defined as 
	$$
	\partial^{\Phi}S:=\overline{S}\setminus \In(S).
	$$
	We remark that the definitions of $\In(S)$ and $\partial^{\Phi}S$ do not depend on the choice of $0<\gamma<\eta$.

Now we consider a global closed cross-section $S$ of time $\eta>0$. Let $\xi>0$ with $\Phi(S\times [-\xi, \xi])=X$. Let $t_S: X\to \R$ be the first time in $S$. Obviously, $t_S$ is bounded from above by $2\xi$ (and from below by $2\eta$ when restricted on $S$). However, the first return time $t_S$ is not continuous on $S$. In general, $t_S$ is a lower semicontinuous positive function. Denote by $\mathcal{C}_S\subset S$ the set of continuity points of $t_S$. The \textit{first return map} $T_S: S\to S$ is defined by  $x\mapsto \Phi(x, t_S(x))$.

The first return map $T_S$ is continuous on $\mathcal{C}_S$ but not on $S$ in general.

\begin{lem}\label{lem:return in the boudary}\label{lem:boudary}
	Let $x\in S\setminus \mathcal{C}_S$. Then $\Phi_{t_S(x)}(x)\in \partial^{\Phi}S$.
\end{lem}
\begin{proof}
	Since $x\in S\setminus \mathcal{C}_S$, there exists a sequence $(x_n)_{n\in \N}$ of points in $S$ converging to $x$ with 
	\begin{equation}\label{eq:return in the boudary}
	t_S(x)<\lim\limits_{n\to \infty} t_S(x_n)\le 2\xi.
	\end{equation}
	Fix a small $0<\gamma<\eta$. By \eqref{eq:return in the boudary}, the point $\Phi_{t_S(x)}(x_n)$ lies out Int$(\Phi_{(-\gamma, \gamma)}(S))$ when $n$ is sufficiently large. So does $\Phi_{t_S(x)}(x)$.
\end{proof}

A cross-section $S$ of $(X, \Phi)$ is called a {\it Poincar\'e cross-section} if the map $\Phi$ is a surjective local homeomorphism from $S\times \R $ to $X$.

\begin{lem}\label{lem:Poincare cross-section}
	Let $S$ be a closed cross-section of topological flow $(X, \Phi)$. Then the following are equivalent.
	\begin{enumerate}
		\item $S$ is a Poincar\'e cross-section.
		\item $S$ is a global closed cross-section with empty flow boundary.
		\item $(X, \Phi)$ is topologically conjugate to the suspension flow $(S_{t_S}, (T_S)_{t_S})$ where $t_S: S\to \R_{>0}$ is the first return time restricted on $S$ and $T_S: S\to S$ the first return map.  
	\end{enumerate}
\end{lem}
\begin{proof}
	The equivalence between (1) and (2) follows by \cite[Lemma 2.9]{burguet2019symbolic}. 
	
	(2) $\Rightarrow$ (3): Since $S$ has the empty flow boundary, by Lemma \ref{lem:boudary}, we see that $t_S$ and $T_S$ are continuous. It follows that the map $(x,t)\mapsto \Phi(x,t)$ is an isomorphism from $(S_{t_S}, (T_S)_{t_S})$ to $(X, \Phi)$, where the inverse map is $x\mapsto (T_S^{-1}(x), t_S^{-1}(x))$.

	(3) $\Rightarrow$ (1): It is not hard to check that $(T_S)_{t_S}$ is a surjective local homeomorphism from $S\times \R $ to $S_{t_S}$.
\end{proof}

\subsection{Inverse limit of dynamical systems}
A pair $(X, G)$ is called a {\it $G$-system} if $X$ is compact space, $G$ is a topological group and $G$ acts on $X$ continuously. For example, $G=\R$ or $\Z$.

Let $\{(X_n, G_n) \}_{n\in \N}$ be a sequence of systems. Suppose that for every pair $m>n$ there exists a factor map $\sigma_{m,n}: X_m\to X_n$ such that for any triple $m>n>l$ the diagram
\begin{equation*}
\xymatrix{
	& (X_m, G_m) \ar[d]^{\sigma_{m,n}}   \ar[ld]_{\sigma_{m,l}}    \\
	(X_l, G_l)  &  \ar[l]^{\sigma_{n,l}} (X_n, G_n) 
}
\end{equation*}
commutes. Set
$$
X=\{x\in \prod_{n\in \N} X_n: \sigma_{m,n}(x_m)=x_n, \forall~m>n \}.
$$
Clearly, the space $X$ is closed in $\prod_{n\in \N} X_n$ and thus compact.
Let $\sigma_n:X\to X_n$ be the projection map for each $n\in \N$. Define the action $G:X\to X$ by
$$
G: (x_n)_{n\in \N} \to (G_nx_n)_{n\in \N}.
$$ 
We call the $G$-system $(X,G)$ the {\it inverse limit} of the family $\{(X_n, G_n) \}_{n\in \N}$ via $\sigma=(\sigma_{m,n})_{m,n\in \N, m>n}$ and write 
$$
(X,G)=\varprojlim (X_n,G_n).
$$
It is clear that if $G_n$ are the same for every $n\in \N$ then $G=G_n$.

\section{Solenoid and its extension}\label{sec:solenoid}

Let $Y_n=[0, n!]$ such that $0$ is identified with $n!$. Let $\Psi^{(n)}_r$ be the flow on $Y_n$ defined by $ x\mapsto x+r \mod n!$ for $r\in \R$ and $x\in Y_n$. The {\it solenoid} is defined by $Y=\varprojlim (Y_n, \Psi^{(n)}).$ Clearly, the solenoid is a minimal flow.
For $x\in Y$, we write $x=(x_n)_{n\ge 1}$ where $x_n$ is the projection of $x$ on $Y_n$.

Let $S_n:=\{x\in Y: x_1=x_2=\dots=x_n=0 \}$. Clearly, the set $S_n$ is closed. Now we state our main result of this section.

\begin{prop}[=Proposition \ref{prop:extension of solenoid, suspension flow 0}]\label{prop:extension of solenoid, suspension flow}
	Let  $(Z, \Phi)$ be an extension of the solenoid $(Y, \Psi)$ via $\pi$. Then $(Z, \Phi)$ is topologically conjugate to the suspension flow $((\pi^{-1}(S_n))_{f_n}, (T_{\pi^{-1}(S_n)})_{f_n} )$ for each $n\ge 1$ where $f_n$ is the constant function $n!$.
\end{prop}

We study firstly the solenoid and then its extensions.

\begin{lem}\label{lem:S Poincare cross-section}
	For every $n\ge 1,$ $S_n$ is a Poincar\'e cross-section of $(Y,\Psi)$.
\end{lem}
\begin{proof}
	Fix $n\ge 1$. By Lemma \ref{lem:Poincare cross-section}, it is sufficient to show that $S_n$ is a global cross-section with empty flow boundary. It is clear that
	\begin{equation*}
	\begin{split}
	\Psi_r(S_n)&=\{x\in Y: x_k \equiv r \mod k!~\text{for}~1\le k\le n  \}\\
	&=\{x\in Y: x_n \equiv r \mod n! \},
	\end{split}
	\end{equation*}
	for $r\in \R$. It follows that $\Psi$ is injective on $S_n\times [0, n!)$ and $\Psi_{[0, n!]}(S_n)=Y$. That is, the set $S_n$ is a global cross-section. Note that $\Psi_{(-\gamma, \gamma)}(S_n)=\{x\in Y: x_n \in [0, \gamma) \cup (n!-\gamma, n!] \}$ which is open (because $[0, \gamma) \cup (n!-\gamma, n!]$ is open in $Y_n$) for $0<\gamma<n!/2$. Thus we have that $\Int(\Psi_{(-\gamma, \gamma)}(S_n))=\Psi_{(-\gamma, \gamma)}(S_n)$ for $0<\gamma<n!/2$ and consequently that $\Int^{\Psi}(S_n)=S_n$. This implies that $\partial^{\Psi}(S_n)=\emptyset$.
\end{proof}

\begin{lem}
	The solenoid $(Y,\Psi)$ is topological conjugate to the suspension flow $((S_n)_{f_n}, (T_{S_n})_{f_n} )$ for each $n\ge 1$ where $f_n$ is the constant function $n!$.
\end{lem}
\begin{proof}
	Let $n\ge 1$. It is clear that the first return time $t_{S_n}\equiv n!$. Combining Lemma \ref{lem:S Poincare cross-section} and Lemma \ref{lem:Poincare cross-section}, we complete the proof.
\end{proof}

Now we consider the extensions of solenoid.

\begin{lem}\label{lem: pi-1 S_n Poincare cross section}
	For every $n\ge 1,$ $\pi^{-1}(S_n)$ is a Poincar\'e cross-section of $(Z,\Phi)$.
\end{lem}
\begin{proof}
	Fix $n\ge 1$. Since $S_n$ is closed and $\pi$ is continuous, we have that $\pi^{-1}(S_n)$ is closed. We claim that $\Phi$ is injective on $\pi^{-1}(S_n)\times [0, n!)$. In fact, if not, then there are $x,x'\in \pi^{-1}(S_n)$ and $t,t'\in [0, n!)$ such that $(x,t)\not=(x',t')$ and $\Phi_t(x)=\Phi_{t'}(x')$. It follows that 
	\begin{equation}\label{eq:1}
	\Phi_{t-t'}(x)=x'.
	\end{equation}
	Then we have $\pi(\Phi_{t-t'}(x))=\pi(x')$, implying that $\Psi_{t-t'}(\pi(x))=\pi(x')$. Since $\pi(x),\pi(x')\in S_n$ and $|t-t'|<n!$, by Lemma \ref{lem:S Poincare cross-section}, we obtain that $t=t'$ and $\pi(x)=\pi(x')$. Combining this with \eqref{eq:1}, we see that $x=x'$, which is a contradiction. Thus $\Phi$ is injective on $\pi^{-1}(S_n)\times [0, n!)$. It follows that $\pi^{-1}(S_n)$ is a closed cross-section.
	
	Since $\Psi_{[0, n!]}(S_n)=Y$, we obtain that
	\begin{equation*}
	\Phi_{[0, n!]}(\pi^{-1}(S_n))=\pi^{-1}(\Psi_{[0, n!]}(S_n))=\pi^{-1}(Y)=Z.
	\end{equation*}
	Since $\Psi_{(-\gamma, \gamma)}(S_n)$ is open and $\pi$ is continuous, we obtain that
	\begin{equation*}
	\Int(\Phi_{(-\gamma, \gamma)}(\pi^{-1}(S_n)))=\Int (\pi^{-1}(\Psi_{(-\gamma, \gamma)}(S_n)))=\pi^{-1}(\Psi_{(-\gamma, \gamma)}(S_n)),
	\end{equation*}
	for $\gamma\in (0, n!/2)$.
	Thus we have that
	\begin{equation*}
	\begin{split}
	\Int(\Phi_{(-\gamma, \gamma)}(\pi^{-1}(S_n))) \cap \pi^{-1}(S_n)&=\pi^{-1}(\Psi_{(-\gamma, \gamma)}(S_n))\cap \pi^{-1}(S_n)
	\\&=\pi^{-1}(\Psi_{(-\gamma, \gamma)}(S_n)\cap S_n)\\
	&=\pi^{-1}(S_n).
	\end{split}
	\end{equation*}
	This means that $\In(\pi^{-1}(S_n))=\pi^{-1}(S_n)$ and consequently that $\Pa(\pi^{-1}(S_n))=\emptyset$. We conclude that $\pi^{-1}(S_n)$ is a Poincar\'e cross-section of $(Z,\Phi)$ by Lemma \ref{lem:Poincare cross-section}.
\end{proof}

Now we give the proof of Proposition \ref{prop:extension of solenoid, suspension flow}.
\begin{proof}[Proof of Proposition \ref{prop:extension of solenoid, suspension flow}]
	Let $n\ge 1$. It is clear that the first return time $t_{\pi^{-1}(S_n)}\equiv n!$. Combining Lemma \ref{lem: pi-1 S_n Poincare cross section} and Lemma \ref{lem:Poincare cross-section}, we complete the proof.
\end{proof}

We remark that the discrete topological system $(\pi^{-1}(S_n), T_{\pi^{-1}(S_n)})$ is the extension of the minimal system $(S_n, T_{S_n})$ via $\pi_{|\pi^{-1}(S_n)}$.

\section{Mean dimension of $\R$-flow}\label{sec:Mean dimension of R-flow}
In this section, we recall several notions related to mean dimension of $\R$-flow.
\subsection{Mean dimension of $\R$-flow}
Let $(X, d)$ be a compact metric space. For $\epsilon>0$ and $Y$ a topological space, a continuous map $f: X\to Y$ is called a {\it $(d, \epsilon)$-embedding} if for any $y\in Y$ we have diam$(f^{-1}(y))<\epsilon$. Note that the identity map from $X$ to itself is a $(d, \epsilon)$-embedding for every $\epsilon>0$.
Denote by $\dim$ the Lebesgue covering dimension.
Define
$$
\wdim_\epsilon(X, d)=\min_K \dim(K),
$$
where $K$ runs over all compact metrizable space such that there is a  $(d, \epsilon)$-embedding $f: X\to K$.

Let $(X, \R)$ be a topological flow. For $R>0$, we define the metric $d_R$ by
$$
d_R(x,y)=\sup_{0\le r\le R} d(r.x, r.y), \forall x,y\in X.
$$
The topology on $X$ is compatible with the metric $d_R$ for $R>0$. The {\it (topological) mean dimension}
of $(X, \R)$ is defined by
$$
\mdim(X, \R)=\lim\limits_{\epsilon\to 0} \lim\limits_{R\to \infty} \frac{\wdim_\epsilon(X, d_R)}{R}.
$$
The limit above exists due to Ornstein-Weiss' Lemma (see \cite[Theorem 6.1]{LindenstraussWeiss2020MeanTopologicalDimension}).

In what follows, we summarize several properties of mean dimension of topological flows which are obtained in \cite{gutmanjin2020realflow}.
\begin{prop}[\cite{gutmanjin2020realflow}, Proposition 2.5]
	Let $(X, \Phi)$ be a topological flow. Then $\mdim(X, \Phi)=\mdim(X, \Phi_1)$.
\end{prop}

\begin{prop}[\cite{gutmanjin2020realflow}, Proposition 2.6]
	Let $(X, \Phi)$ be a topological flow. If the topological entropy $h(X, \Phi)$ is finite, then $\mdim(X, \Phi)=0$.
\end{prop}

Combining \cite[Theorem 4.3]{L99} and \cite[Proposition 3.3]{gutmanjin2020realflow}, we have the following proposition.  
\begin{prop}\label{prop:mdim extension}
	Let $(Z, T)$ be an extension of nontrivial minimal system.  Let $f: X\to \{1\}$ be the constant function. Then $\mdim(Z_f, T_f)=\mdim(Z, T)$.
\end{prop}

\subsection{$\R$-shift on $B_1(V[a,b])$}

A function $f:\R \to \C$ is called a {\it Schwartz function} (or {\it rapidly decreasing function}) if it is infinitely differentiable and satisfies
$$
\norm{f}_{\alpha,\beta}:=\sup_{x\in \R} \left|x^\alpha f^{(\beta)}(x) \right|<\infty,
$$
for all $\alpha, \beta\in \N$. The {\it Schwartz space} $\SS(\R)$ is defined as the space of all Schwartz functions. Recall that a {\it tempered distribution} on $\R$ is a continuous linear functional on $\SS(\R)$. We remarks that bounded continuous functions are tempered distributions.

For $L^1$-function $f:\R\to \C$, the Fourier transformations of $f$ are defined by
$$
\FF(f)(\xi)=\int_{\R} f(x)e^{-2\pi i \xi x} dx,~ \FFi(f)(\xi)=\int_{\R} f(x)e^{2\pi i \xi x} dx.
$$
Moreover, if additionally $\FF(f)\in L^1(\R)$ (resp. $\FFi(f)\in L^1(\R)$), then $\FFi(\FF(f))=f$ (resp. $\FF(\FFi(f))=f$). 

The Fourier transformations $\FF$ and $\FFi$ of tempered distribution $\phi$ are defined by the tempered distributions satisfying that
$$
\langle \FF(\phi), \Psi \rangle=\langle \phi, \FFi(\Psi) \rangle, \langle \FFi(\phi), \Psi \rangle=\langle \phi, \FF(\Psi) \rangle, \forall\Psi\in \SS(\R).
$$
For $a<b$ and tempered distribution $\phi$, we say that the {\it support} of $\phi$ is contained in $[a,b]$ and write supp$(\phi)\subset [a,b]$ if $\langle \phi, \Psi \rangle=0$ for any $\Psi\in \SS(\R)$ with supp$(\Psi)\cap [a,b]=\emptyset$.



Let $a<b$ be real numbers. We define $B_1(V[a,b])$ as the space of all continuous functions $f:\R \to \C$ satisfying the following conditions:
\begin{itemize}
	\item supp$\FF(f)\subset [a,b]$.
	\item $\norm{f}_{\infty}\le 1$.
\end{itemize}
The space $B_1(V[a,b])$ is metric and compact under the distance $d$ induced by the distance on $C(\R, \mathbb{D})$, that is,
$$
d(f,g)=\sum_{n=1}^{\infty} \frac{\norm{f\cdot1_{[-n,n]}-g\cdot1_{[-n,n]}}_{\infty}}{2^n},~\text{for}~f,g\in B_1(V[a,b]),
$$
where $\mathbb{D}$ is the unit disk in $\C$.
The topology induced by the distance $d$ is compatible with the standard topology of tempered distribution. 

For $r\in \R$, we denote by $\tau_r$ the {\it translation} by $r$, that is, $\tau_r(f)(x)=f(x+r)$ for function $f:\R\to \C$ and $x\in \R$. Since $\FF(\tau_rf)(\xi)=e^{2\pi i r\xi}\FF(f)(\xi)$, one sees that supp$(\FF(\tau_rf))=\text{supp}(\FF(f))$. It follows that $(\tau_r)_{r\in \R}$ induces the $\R$-shift on $B_1(V(a,b))$, denoted by $(B_1(V[a,b]), \R)$.

\begin{lem}[\cite{gutman2019application}, Footnote 4]
	$\mdim(B_1(V[a,b]), \R)=b-a$.
\end{lem}

\subsection{$\R$-shift on $\BB(a)$}\label{sec: BB(a)}
Let $a>0$. We define $\BB(a)$ as the space of all continuous functions $f:\R \to \R$ satisfying the following conditions:
\begin{itemize}
	\item supp$\FF(f)\subset [-\frac{a}{2},\frac{a}{2}]$.
	\item $\norm{f}_{\infty}\le 1$.
\end{itemize}
Similarly to the space $B_1(V[b_1,b_2])$ for $b_1<b_2$,  the space $\BB(a)$ is metric and compact under the distance $d$ induced by the distance on $C(\R, [-1,1])$.
The translation $(\tau_r)_{r\in \R}$ induces the $\R$-shift on $\BB(a)$, denoted by $(\BB(a), \R)$.
\begin{lem}[\cite{gutman2019application}, Footnote 4]
	$\mdim(\BB(a), \R)=2a$.
\end{lem}

The relation between $B_1(V[a,b])$ and $\BB(2b)$ is shown as follows.
\begin{lem}
	Let $b>a>0$. Then there is an embedding from $B_1(V[a,b])$ to $\BB(2b)$.
\end{lem}
\begin{proof}
	It is easy to check that $\phi\mapsto \frac{1}{2}(\phi+\overline{\phi})$ is the embedding.
\end{proof}
\section{Strongly embedding Theorem for discrete systems}\label{sec:Strongly embedding Theorem for discrete systems}

Gutman, Qiao and Tsukamoto \cite{gutman2019application} proved the embedding theorem as follows. In fact, they proved the result for $\Z^k$-action and we state it for $\Z$-action in current paper for sake of simplicity.
\begin{thm}[\cite{gutman2019application}, Main Theorem 2]\label{thm: gutman}
	If a topological dynamical system $(X,T)$ satisfies the marker property and $\mdim(X,T)<\frac{a}{2}$
	then we can embed it in the shift on $\BB(a)$.
\end{thm}

For a discrete system $(X,T)$ and a topological flow $(Y, \Phi)$, we say that $(X,T)$ is {\it strongly embedded} in $(Y,\Phi)$ if $(X,T)$ is embedded in $(Y, \Phi_1)$ via $h$ and $h$ satisfies that if there exists $x,x'\in X$ and $r\in \R$ such that $\Phi_r(h(x))=h(x')$ then 
$r\in \Z$ and $x'=T^rx$. Actually, it means that different $\Z$-orbits are embedded in different $\R$-orbits.

Building on Theorem \ref{thm: gutman}, we show the following theorem.
\begin{thm}[=Theorem \ref{main thm 2 0}]\label{main thm 2}
	If a topological dynamical system $(X,T)$ satisfies the marker property and $\mdim(X,T)<\frac{a}{2}$
	then we can strongly embed it in the shift on $\BB(a)$.
\end{thm}
Actually, the above result also holds for $\Z^k$-action by the same proof. Since we consider $\R$-action (as well as $\Z$-action) in our paper, we only give the proof for $\Z$-action in what follows.

The proof of Theorem \ref{main thm 2} follows by the main strategy of Theorem \ref{thm: gutman}. The key point of our proof is to make the parameter $r_1$ vary with respect to the parameter $L$ (see Section \ref{sec:Tiling-Like Band-Limited Map}). The proof itself of Theorem \ref{thm: gutman} is highly sophisticated and deeply technical. We will not repeat several proofs of technical lemmas and refer to \cite{gutman2019application} for detail.

\subsection{Tilling}\label{subsec:tilling}

A collection $\mathcal{W}=\{W_n\}_{n\in \Z^k}$ of sets is said to be a {\it tiling} of $\R^k$  if $\cup_{n\in \Z^k}W_n=\R^k$ and the Lebesgue measure of $W_n\cap W_m$ vanishes for distinct $n,m\in \Z^k$. Moreover, it called a {\it convex tiling} of $\R^k$ if $W_n$ are convex for all $n\in \Z^k$.

Let $(X,T)$ be a discrete topological system with marker property. Let $M>0$ be an integer. Then there exists an open set $U\subset X$ such that
$$
U\cap T^{-n}U=\emptyset~(\text{for}~0<|n|<N)~\text{and}~X=\cup_{n\in \Z} T^{n}U.
$$
Then there exists an integer $M_1>M$ and a compact set $F\subset U$ such that $X=\cup_{|n|<M_1} T^nF$. Choose a continuous map $\phi: X\to [0,1]$ satisfying that supp$(\phi)\subset U$ and $h=1$ on $F$. Let
$$
\mathcal{S}(x)=\left\{\left(n, \frac{1}{\phi(T^nx)}\right): n\in \Z, \phi(T^nx)>0  \right\}
$$
be the discrete set in $\R^2$. The {\it Voronoi tiling} $\mathcal{V}(x)=\{V_0(x,n)\}_{n\in \Z}$ is defined as follows: if $\phi(T^nx)=0$ then $V_0(x,n)=\emptyset$; if $\phi(T^nx)>0$ then 
$$
V_0(x,n)=\left\{u\in \R^2: \left|u-\left(n, \frac{1}{\phi(T^nx)}\right)\right|\le \left|u-p\right|, \forall p\in \mathcal{S}(x)\right\}.
$$
Clearly, the sets $V_0(x,n)$ form a tiling of $\R^2$. Let $\pi: \R^2\to \R$ be the projection on the first coordinate. Let $H=(M_1+1)^2$ and
$$
W_0(x,n)=\pi \left(V_0(x,n)\cap (\R\times \{-H \}) \right).
$$
Then the sets $W_0(x,n)$ form a tiling of $\R$. Moreover, the tiling $\{W_0(x,n) \}_{n\in \Z}$ is $\Z$-equivariant, i.e. $W_0(T^nx,m)=-n+W_0(x, n+m)$.

\begin{lem} [\cite{gutman2019application}, Claim 5.11]
Let $x\in X$ and $n\in \Z$ with $h(T^nx)>0$. Then the following hold.	
\begin{itemize}
	\item [(1)] $V_0(x,n)$ contains the ball $B_{M/2}(n, 1/h(T^nx))$.
	\item [(2)] $W_0(x,n)$ is contained in $B_{M_1+1}(n)$.
	\item [(3)] If $W_0(x,n)\not=\emptyset$ then $h(T^nx)>1/2$.
\end{itemize}
\end{lem}

Let $\epsilon>0$. Pick $1<c<\frac{1}{1-\epsilon}$. Let $M_2=\frac{(c-1)H}{H+2}M$. Notice that $M_2\approx (1-c^{-1})M$ as $M_1$ is sufficient large. Using the same method of  \cite[Claim 5.12]{gutman2019application}, we obtain the following lemma.
\begin{lem}\label{lem:M_2}
	Let $x\in X$ and $n\in \Z$ with $h(T^nx)>0$ and $W_0(x,n)\not=\emptyset$. Then $W_0(x,n)$ contains the ball of radius $M_2$.
\end{lem}

\subsection{Tiling-Like Band-Limited Map}\label{sec:Tiling-Like Band-Limited Map}
Let $a>0$ and $b=a+\delta/2$. Let $L>0$. For $r>0$ and $u\in \C$, we define $D_r(u)$ as the closed disk centered at $u$ of radius $r$ in $\C$. Define 
$$
\Theta_L: \C \to \C, z\mapsto e^{\pi i bz}\sin(\frac{\pi z}{L}). 
$$
Pick $r_1=r_1(L)$ such that
\begin{itemize}
	\item $0<r_1<\min\{\frac{1}{16}, \frac{1}{L}\}$.
	\item For $|z|<r_1$, we have that
	$$
	\pi \left| b\sin(\frac{\pi z}{L})+\frac{1}{L}\cos(\frac{\pi z}{L})  \right|>\frac{3}{L}.
	$$ 
\end{itemize}
We remark here that the choice of $r_1$ is the key point in our proof. In \cite[
Notation 5.3]{gutman2019application}, they choose $r_1$ independent of $L>1$. Here, we choose $r_1\to 0$ as $L\to \infty$.

Let 
$$
\Omega=\{z\in \C: |\text{Im}(z)|\le 1 \}.
$$
Define 
$$
\theta_L=\min\left\{\frac{9}{16L}, \inf\{|\Theta_L(z)|:z\in \Omega \setminus \cup_{n\in \Z} D_{r_1}(Ln)  \} \right\}.
$$
We choose a Schwartz function $\chi_1: \R \to \R$ satisfying that
\begin{itemize}
	\item supp$\FF(\chi_1)\subset B_{\delta/8}(0)$.
	\item $\int_\R \chi_1(x)dx=1$.
\end{itemize}

Let $E=E(L, \theta_L)>0$ such that for all $z\in \Omega$, we have 
$$
\norm{(\Theta_L)_{|\Omega}}_{\infty} \int_{\R\setminus B_E(\text{Re}(z)) } |\chi_1(z-t)|dt<\frac{\theta_L}{2}.
$$

Recall that the tilling $\{W_0(x,n)\}_{n\in \Z} $ was defined in Section \ref{subsec:tilling}. For $x\in X$, define $\Phi(x): \C \to \C$ as
$$
\Phi(x)(z)=\sum_{n\in \Z} \Theta_L(z-n) \int_{W_0(x,n)} \chi_1(z-t)dt.
$$
\begin{lem}[\cite{gutman2019application}, Lemma 5.9, Lemma 5.14]\label{lem:Phi property}
	For $x\in X$, the function $\Phi(x)$ satisfies the following properties.
	\begin{itemize}
		\item [(1)] $\norm{\Phi(x)_{|\R}}_{{\infty}}\le K_1$, where $K_1=\int_{\R} |\chi_1(t)|dt$.
		\item [(2)] If $L>4/\delta$ then the support of $\FF\left( \Phi(x)_{|\R} \right)$ is contained in $(a/2, a/2+\delta/2)$.
		\item [(3)] $\Phi(x)$ is $\Z$-equivariant, i.e. $\Phi(T^nx)(z)=\Phi(x)(z+n)$.
		\item [(4)] If $x_n$ converges to $x$ in $X$ then $\Phi(x_n)$ converges to $\Phi(x)$ uniformly over every compact subset of $\C$.
		\item [(5)] For $z\in \Omega$, if Re$(z)\in \text{Int}_E W_0(x,n)$ and $\Phi(x)(z)=0$ then there exists $m\in \Z$ satisfying $z\in D_{r_1}(n+Lm)$, where $\text{Int}_E W_0(x,n)=\{y\in W_0(x,n): B_E(y)\in W_0(x,n)  \}$.
		\item [(6)] If $n,m\in \Z$ with $n+Lm\in \text{Int}_{E+1} W_0(x,n)$, then there exists $z\in D_{r_1}(n+LM)$ such that $\Phi(x)(z)=0$.
	\end{itemize}
\end{lem}

\subsection{Main proposition}
As in \cite{gutman2019application}, we pick $L$ sufficient large (i.e. it satisfies \cite[Condition 5.13]{gutman2019application}) and $M$ large (i.e. it satisfies \cite[Equation (5.9)]{gutman2019application}). Moreover, we assume $M_2=M_2(M)$ satisfies the following condition.

\begin{con}\label{cond:M_2}
	$M_2>4L+E+1$.
\end{con}
Combing this condition with Lemma \ref{lem:M_2} and Lemma \ref{lem:Phi property} (5) $\&$ (6), we obtain the following lemma.

\begin{lem}\label{lem:zero and nonzero}
	 If $W(x,n)\not=\emptyset$ then
	 \begin{itemize}
	 	\item $
	 	\text{there exists}~ m\in \Z ~\text{such that}~n+Lm, n+L(m+1)\in \text{Int}_E W_0(x,n);
	 	$
	 	\item there exists $z_1\in D_{r_1}(n+Lm)$  and $z_2\in D_{r_1}(n+L(m+1))$ such that $\Phi(x)(z_1)=\Phi(x)(z_2)=0$;
	 	\item $\Phi(x)(z)\not=0$ for all $z$ with Re$(z)\in (n+Lm+r_1, n+L(m+1)-r_1)$.
	 \end{itemize}
\end{lem}

The main proposition is shown as follows. We remark that comparing to \cite[Proposition 3.1]{gutman2019application}, the statement (3) in the following proposition is crucial to prove Theorem \ref{main thm 2}.
\begin{prop}\label{main prop 2}
	Let $(X,T)$ be a dynamical system  satisfying the marker property and $\mdim(X,T)<\frac{a}{2}$. Let $f: X\to \BB(a)$ be a $\Z$-equivalent continuous map. There exists a $\Z$-equivalent continuous map $g: X \to \BB(a+\delta)$ such that 
	\begin{itemize}
		\item [(1)] $\norm{g(x)-f(x)}<\delta$ for all $x\in X$.
		\item [(2)] $g$ is a $(d, \delta)$-embedding.
		\item [(3)] If there exists $x,y\in X$ and $r\in [-1/2,1/2]$ such that $\tau_r(g(y))=g(x)$ then $|r|\le 2r_1$.
	\end{itemize}
\end{prop}
\begin{proof}
	(1) and (2) follows by \cite[Proposition 3.1]{gutman2019application} where 
	$$
	g(x)=g_1(x)+g_2(x)
	$$
	satisfying that $g_2(x)=\frac{\delta}{2K_1}\text{Re}(\Phi(x)_{|\R})$ and supp$(\FF(g_1(x)))\cap$supp$(\FF(g_2(x)))=\emptyset$. It remains to prove (3). Notice that if $g(x)=g(y)$ then $g_2(x)=g_2(y)$ implying that $\Phi(x)=\Phi(y)$\footnote{The supports of $\FF(\Phi(x))$ and $\FF(\overline{\Phi(x)})$ are distinct.}. Thus (3) follows by Lemma \ref{lem:Phi 2}.
\end{proof}

\begin{lem}\label{lem:Phi 2}
	If there exists $x,y\in X$ and $r\in [-1/2,1/2]$ such that $\tau_r(\Phi(y))=\Phi(x)$ then $|r|\le 2r_1$.
\end{lem}
\begin{proof}
	Without loss of generality, we assume $r\ge 0$.
	Due to Lemma \ref{lem:M_2} and Condition \ref{cond:M_2}, we see that if $W_0(x,n_x)\cap W_0(y,m)\not=\emptyset$ and $W_0(y,m_1)$, $W_0(y,m)$ and $W_0(y,m_2)$ are successive nonempty intervals then there exists $m_y\in\{m_1, m, m_2\}$ satisfying that
	\begin{equation}\label{eq:Phi 2.1}
	\left| W_0(x,n_x)\cap W_0(y,m_y) \right|\ge\frac{2M_2-2E}{2}>4L.
	\end{equation}
	Let $n\in \Z$ with $W_0(x,n)\not=\emptyset$. By Condition \ref{cond:M_2} and Equation \eqref{eq:Phi 2.1}, there exists $n_y\in \Z$ such that $\left| W_0(x,n)\cap W_0(y,m_y) \right|>4L$.  Let $m,z_1$ be as Lemma \ref{lem:zero and nonzero}, that is, $\Phi(x)(z_1)=0$ and $z_1\in D_{r_1}(n+Lm)$, and satisfy that $[n+Lm-r_1, n+L(m+1)+r_1]\subset W_0(y, n_y)$. If $r\in (2r_1, 1)$ then
	$$
	\Phi(y)(z_1+r)=\Phi(x)(z_1)=0.
	$$
	However $D_{r_1}(z_1+r)\subset D_{2r_1}(n+Lm+r)$ which doesn't contain the integer points on real axis. This is a contradiction to Lemma \ref{lem:Phi property} (5) (with respect to the tiling $\{W_0(y,n)\}_{n\in \Z}$). Therefore we conclude that $r\le 2r_1$. 
\end{proof}

\subsection{Proof of Theorem \ref{main thm 2}}

We first show the lemma which roughly says that the collection of functions satisfying Proposition \ref{main prop 2} (3) is open, and then use this lemma to prove Theorem \ref{main thm 2}.
\begin{lem}\label{lem:close to g}
	Let $a>1$ and $(X, T)$ be a discrete dynamical system. Let $r_1, \eta>0$. Let $g: X\to \BB(a)$ be a $\Z$-equivariant continuous map satisfying that if there exists $x,y\in X$ and $r\in [-1/2,1/2]$ such that $\tau_r(g(y))=g(x)$ then $|r|\le 2r_1$. Then there exists $\epsilon>0$ such that for any $\Z$-equivariant continuous map $h: X \to \BB(a+\eta)$ with $\sup_{x\in X}\norm{h(x)-g(x)}_{\R, \infty}<\epsilon$ if there exists $x,y\in X$ and $r\in [-1/2,1/2]$ such that $\tau_r(h(y))=h(x)$ then $|r|< 3r_1$.
\end{lem}
\begin{proof}
	We prove it by contradiction. Assume there exist $r_n\in [-1/2, -3r_1]\cup [3r_1, 1/2]$, $\Z$-equivariant continuous maps $h_n: X\to \BB(a+\eta)$ and points $x_n, y_n\in X$ such that $x_n\to x$, $y_n\to y$, $r_n\to r\in [-1/2, -3r_1]\cup [3r_1, 1/2]$ as $n\to \infty$,
	\begin{equation*}
	\tau_{r_n}(h_n(y_n))=h_{r_n}(x_n), \forall n\in \N,
	\end{equation*}
	and
	\begin{equation*}
	\sup_{x\in X}\norm{h_n(x)-g(x)}_{\R, \infty}<\frac{1}{n}, \forall n\in \N.
	\end{equation*}
	Taking $n$ tend to $\infty$, we obtain that
	\begin{equation*}
	\tau_r(g(y))=g(x),
	\end{equation*}
	which is a contradiction.
\end{proof}

Now we prove the main result in this section.
\begin{proof}[Proof of Theorem \ref{main thm 2}]
	Without loss of generality, we assume diam$(X,d)<1$. Pick a strictly increasing sequence $\{a_i\}_{i\in \N}$ such that 
	$$
	0<a_i<a~\text{and}~\mdim(X)<\frac{a_i}{2}, \forall i\in \N.
	$$
	We will inductively define positive numbers $\epsilon_n, r_n$ and $\Z$-equivariant continuous $(d, 1/n)$-embedding map $h_n: X\to \BB(a_n)$ such that if there exists $x,y\in X$ and $r\in [-1/2,1/2]$ such that $\tau_r(h_n(y))=h_n(x)$ then $|r|\le 2r_n$. For $n=1$, we define
	$$
	\epsilon_1=1, h_1(x)=0 (\forall x\in X)~\text{and}~r_1=\frac{1}{4}.
	$$
	Clearly, $h_1$ is a $(d, 1)$-embedding map and satisfies that if there exists $x,y\in X$ and $r\in [-1/2,1/2]$ such that $\tau_r(h_1(y))=h_1(x)$ then $|r|\le 1/2$. Suppose $\epsilon_n, r_n$ and $h_n$ are well defined. Since $h_n$ is a $(d, 1/n)$-embedding map, there exists $0<\epsilon_n'<\epsilon_n/2$ such that if a $\Z$-equivariant continuous map $h: X \to \BB(a)$ satisfies $\sup_{x\in X}\norm{h_n(x)-h(x)}_{\infty}<\epsilon_n'$, then $h$ is also a $(d, 1/n)$-embedding map. Moreover, by Lemma \ref{lem:close to g}, there exists $\epsilon_n''>0$ such that for any $\Z$-equivariant continuous map $h: X \to \BB(a)$ with $\sup_{x\in X}\norm{h_n(x)-h(x)}_{\R, \infty}<\epsilon_n''$ if there exists $x,y\in X$ and $r\in [-1/2,1/2]$ such that $\tau_r(h(y))=h(x)$ then $|r|< 3r_n$. Let $\epsilon_{n+1}=\min\{\epsilon_n', \epsilon_n'' \}$. Let $$\delta=\min\left\{\frac{1}{n+1}, \frac{\epsilon_{n+1}}{2}, a_{n+1}-a_n \right\}.$$ We choose a large number $L$ satisfying Condition \ref{cond:M_2} and $r_{n+1}=r_{n+1}(L)$ as in Section \ref{sec:Tiling-Like Band-Limited Map} with $0<r_{n+1}<r_n/2$. 
	By Proposition \ref{main prop 2}, we can find a $\Z$-equivariant continuous map $h_{n+1}: X\to \BB(a_n)$ such that 
	\begin{itemize}
		\item [(1)] $\norm{g(x)-f(x)}_{\infty}<\delta$ for all $x\in X$.
		\item [(2)] $g$ is a $(d, \delta)$-embedding.
		\item [(3)] If there exists $x,y\in X$ and $r\in [-1/2,1/2]$ such that $\tau_r(g(y))=g(x)$ then $|r|\le 2r_{n+1}$.
	\end{itemize}
	Therefore, $\epsilon_{n+1}, r_{n+1}$ and $h_{n+1}$ are constructed. We will show that the limit of $h_n$ exists and is what we desire. A simple computation shows that for $n>m\ge 1$ and $x\in X$,
	\begin{equation*}
	\begin{split}
	\norm{h_n(x)-h_m(x)}_{\infty}
	&\le \sum_{\ell=m}^{n-1} \norm{h_\ell(x)-h_{\ell+1}(x)}_{\infty}\\
	&< \sum_{\ell=m}^{n-1} \frac{\epsilon_{\ell+1}}{2}<\sum_{\ell=1}^{\infty} 2^{\ell} \epsilon_{m+1}=\epsilon_{m+1},
	\end{split}
	\end{equation*}
	which tends to $0$ as $m\to \infty$. It follows that $\{h_n\}_{n\in \N}$ is a Cauchy sequence and thus the limit exists, say $h: X\to \BB(a)$. Furthermore, it is $\Z$-equivariant continuous and satisfies that
	$$
	\sup_{x\in X}\norm{h_n(x)-h(x)}_{\infty}<\epsilon_n, \forall n\in \N.
	$$
	Therefore, by the definition of $\epsilon_n$, we obtain that $h$ is a $(d, 1/n)$-embedding and satisfies that if there exists $x,y\in X$ and $r\in [-1/2,1/2]$ such that $\tau_r(h(y))=h(x)$ then $|r|< 3r_n$ for all $n\in \N$. It follows that $h$ is an embedding and if there exists $x,y\in X$ and $r\in [-1/2,1/2]$ such that $\tau_r(h(y))=h(x)$ then $r=0$. It remains to show that $h$ is a strongly embedding. Suppose there exists $x,y\in X$ and $r\in \R$ such that $\Phi_r(h(y))=h(x)$. Let $$m\in \{k\in \Z: |r-k|\le |r-n|, \forall n\in \Z \}.$$ Then $|r-m|\le [-1/2, 1/2]$. It follows that
	$$
	h(x)=\tau_r (h(y))=\tau_{r-m}(h(T^my)),
	$$
	implying that $r=m$ and consequently $x=T^my$. This completes the proof.
\end{proof}

\section{Embedding Theorem for suspension flow}\label{sec:Embedding Theorem for suspension flow}




In this section, we study the embedding of suspension flows. The main result is as follows.
\begin{thm}\label{thm:embedding suspension flow}
	Let $(Z, T)$ be a discrete topological system. Let $f: Z\to \{1\}$ be the constant function. Let $a>0$. Then the discrete topological system $(Z,T)$ can be strongly embedded in $B(a)$ if and only if the suspension flow $(Z_f, T_f)$ can be embedded in $B(a)$.
\end{thm}
\begin{proof}
	Suppose $H: Z_f\to \BB(a)$ is an embedding. Then it is easy to check that $(Z,T)$ can be strongly embedded in $\BB(V[a,b])$ via $H_{|Z}: Z\to \BB(a)$.
	
	Now suppose $h: Z\to \BB(a)$ is a strong embedding. Define $h_{f}: Z_f\to \BB(a)$ by $h_{f}(x, t)=\tau_t(h(x))$. We claim that $h_{f}$ is an embedding. In fact, if there exists two distinct points $(x,t), (y,s)\in Z_f$ such that $h_{f}(x,t)=h_{f}(y,s)$, then we have 
	$$
	\tau_t(h(x))=\tau_s(h(y)).
	$$
	It follows that $t-s\in \Z$ and $T^{t-s}x=y$, implying that $t=s$ and $x=y$. This is a contradiction.
	we complete the proof.
\end{proof}

We remark that Theorem \ref{thm:embedding suspension flow} may fail when $f$ is not the constant function $1$. This is mainly because Proposition \ref{prop:mdim extension} may not hold for such case, that is, the mean dimension of suspension flow under non-constant function may no longer equal the one of discrete topological flow.

Applying Theorem \ref{main thm 2 0} to Theorem \ref{thm:embedding suspension flow}, we have the following corollary.
\begin{cor}\label{cor:embedding suspension flow}
	Let $(Z, T)$ be a discrete topological system. Suppose that $(Z, T)$ satisfies the marker property and has $\mdim(Z,T)<\frac{a}{2}$. Let $f: Z\to \{1\}$ be the constant function. Then the suspension flow $(Z_f, T_f)$ can be embedded in $\BB(a)$.
\end{cor}

Now we have a new proof of \cite[Theorem 5.1]{gutmanjin2020realflow} as follows.
\begin{cor}[=Corollary \ref{cor:embedding gutman}]
	Let $(X,\Phi)$ be an extension of the solenoid. Suppose that $\mdim(X, \Phi)<a/2$. Then $(X,\Phi)$ can be embedded in $\BB(a)$.
\end{cor}
\begin{proof}
	By Proposition \ref{prop:extension of solenoid, suspension flow}, $(X,\Phi)$ is topologically conjugate to a suspension flow $(Z_f, T_f)$ for some discrete topological system $(Z, T)$ where $f$ is the constant function $1$. By Proposition \ref{prop:mdim extension}, we see that $\mdim(Z,T)=\mdim(X, \Phi)<b-a$. By remark at the end of Section \ref{sec:solenoid}, $(Z, T)$ is an extension of nontrivial minimal system and thus satisfies marker property. By Theorem \ref{main thm 2 0},  $(Z, T)$ can be strongly embedded in $\BB(a)$. It follows from Theorem \ref{thm:embedding suspension flow} that $(X,\Phi)$ can be embedded in $\BB(a)$.
\end{proof}

\section{Further discussion}\label{sec:Further discussion}
In this section, we discuss several open problems. 
Lindenstrauss and Tsukamoto \cite{lindenstraussTsukamoto2014mean} conjectured that
\begin{conj}\label{conj:1}
	Let $(X, T)$ be a topological dynamical system. Suppose that $\rm{mdim}(X, T)<\frac{d}{2}$ and
	$$
	\frac{\rm{dim}(\{x: T^nx=x\})}{n}<\frac{d}{2}~\text{for all}~n\ge 1.
	$$
	Then there is an embedding from $(X,T)$ into $(([0,1]^d)^{\Z}, \sigma)$.
\end{conj}
This conjecture holds generically \cite[Appendix A]{gutman2018embedding}, but it is still widely open in general. Since there is the embedding from $(\BB(a), \Z)$ to $(([0,1]^d)^{\Z}, \sigma)$ for $a<d$ by sampling theory (see also \cite[Lemma 1.5]{gutman2020embedding}), 
Gutman, Qiao and Tsukamoto \cite{gutman2019application} conjectured that
\begin{conj}\label{conj:2}
	Let $(X, T)$ be a topological dynamical system. Suppose that $\rm{mdim}(X, T)<\frac{a}{2}$ and
	$$
	\frac{\rm{dim}(\{x: T^nx=x\})}{n}<\frac{a}{2}~\text{for all}~n\ge 1.
	$$
	Then there is an embedding from $(X,T)$ into $\BB(a)$.
\end{conj}

We have introduced the strong embedding in the current paper. We then conjecture that
\begin{conj}\label{conj:3}
	Let $(X, T)$ be a topological dynamical system. Suppose that $\rm{mdim}(X, T)<\frac{a}{2}$ and
	$$
	\frac{\rm{dim}(\{x: T^nx=x\})}{n}<\frac{a}{2}~\text{for all}~n\ge 1.
	$$
	Then there is a strong embedding from $(X,T)$ into $\BB(a)$.
\end{conj}

The relation of above conjectures is as follows:

\medskip
\centerline{Conjecture \ref{conj:3} $\Rightarrow$ Conjecture \ref{conj:2} $\Rightarrow$ Conjecture \ref{conj:1}.
}
\medskip

Finally, using the same idea in \cite[Appendix A]{gutman2018embedding}, we can show that the conjecture \ref{conj:3} holds generically (but it is still open in general).
\begin{thm}
	The conjecture \ref{conj:3} holds generically.
\end{thm}
\begin{proof}
	Due to \cite[Corollary 3.6]{hochman2006genericity} and \cite[Appendix A]{gutman2018embedding}, the zero-dimensional aperiodic system is generic. Since the zero-dimensional aperiodic system has mean dimension zero and satisfies marker property (see \cite{downarowicz2006minimal}), by Theorem \ref{main thm 2 0}, we obtain that Conjecture \ref{conj:3} holds generically.
\end{proof}


\section*{Acknowledgement}
The author is indebted to Yonatan Gutman for carefully reading the manuscript and for giving constructive comments.
The author also wishes to thank David Burguet and Masaki Tsukamoto for valuable remarks.

\bibliographystyle{alpha}
\bibliography{universal_bib}

\end{document}